\documentclass[12 pt]{article}
\usepackage{stmaryrd}
\usepackage{times}
\usepackage{booktabs}
\usepackage{subfigure}

\usepackage{pifont}
\usepackage{floatrow}
\usepackage{caption}
\usepackage{mathrsfs}
\usepackage[fleqn]{amsmath}
\usepackage{amsfonts,amsthm,amssymb,mathrsfs,bbding}
\usepackage{txfonts}
\usepackage{graphics,multicol}
\usepackage{graphicx}
\usepackage{color}
\usepackage{caption}
\usepackage{indentfirst}
\usepackage{cite}
\usepackage{latexsym,bm}
\usepackage{enumerate}
\usepackage{epsfig}

\usepackage{diagbox}
\evensidemargin=\oddsidemargin\topmargin=-1.5cm

\newtheorem{thm}{Theorem}[section]

\newtheorem{lem}[thm]{Lemma}
\newtheorem{cor}[thm]{Corollary}

\newtheorem{remark}[thm]{Remark}

\theoremstyle{definition}
\newtheorem{defi}{Definition}[section]

\allowdisplaybreaks\allowdisplaybreaks[4]

\addtocounter{section}{0}
\begin{document}
\title{The Kirchhoff Index of Enhanced Hypercubes\footnote{Supported
by the National Natural Science Foundation of China (Grant Nos. 11671344, 11701492).}}
\author{{\small Ping Xu, \ \ Qiongxiang Huang\footnote{
Corresponding author.
Email: huangqx@xju.edu.cn, huangqxmath@163.com}}\\[2mm]\scriptsize
College of Mathematics and Systems Science,
\scriptsize Xinjiang University, Urumqi, Xinjiang 830046, P.R.China}
\date{}
\maketitle {\flushleft\large\bf Abstract } Let $\{e_{1},\ldots,e_{n}\}$ be the standard basis of abelian group $Z_{2}^{n}$, which can be also viewed as a linear space of dimension $n$ over the Galois filed $F_{2}$, and $\epsilon_{k}=e_k+e_{k+1}+\cdots+e_n$ for some $1\le k\le n-1$. It is well known that the so called enhanced hypercube $Q_{n, k}(1\le k \le n-1)$ is just the Cayley graph $Cay(Z_{2}^{n},S)$ where $S=\{e_{1},\ldots, e_{n},\epsilon_{k}\}$. In this paper, we obtain the spectrum of $Q_{n, k}$, from which we give an exact formula of the Kirchhoff index of the enhanced hypercube $Q_{n, k}$. Furthermore, we prove that, for a given $n$, $Kf(Q_{n, k})$ is increased with the increase of $k$. Finally, we get
$\lim\limits_{n\to\infty}\frac{Kf(Q_{n, k})}{\frac{2^{2n}}{n+1}}=1$.
\begin{flushleft}
\end{flushleft}

\section{Introduction}\label{se-1}
An interconnection network is always represented by a graph $\Gamma=(V,E)$, where $V$ denotes the node set and $E$ denotes the edge set. Now various interconnections are proposed. The hypercubes network obtained considerable attention in virtue of its perfect properties, such as symmetry, regular structure, strong connectivity, and small diameter \cite{Bossard, Fink}. An $n$-dimensional hypercube denoted by $Q_{n}$ has $2^{n}$ vertices, and the vertex set is $V(Q_{n})=\{x_{1}x_{2}\cdots x_{n} \mid x_{i}=0~~ or~~ 1, i=1,2,\ldots, n\}$. Two vertices $X=x_{1}x_{2}\cdots x_{n}$ and $Y=y_{1}y_{2}\cdots y_{n}$ are adjacent if and only if there exists  $1\le i \le n$, such that $x_{i}=\overline{y_{i}}$, where $\overline{y_{i}}$ denoted the complement of binary digit $y_{i}$, and $x_{j}=y_{j}$ for all $j\ne i$. As the importance of the hypercubes networks, many variants of it were presented, among which, for instance, are enhanced hypercube, augmented hypercube, folded hypercube \cite{Tzeng, El-Amawy, Choudum}. The $n$-dimensional enhanced hypercube is one of the important variants of hypercube introduced by Tzeng in \cite{Tzeng} which is defined as follows.
\begin{defi}\label{defi-1}
For $n\ge 2$ and $1\le k\le n-1$, the enhanced hypercube $Q_{n, k}=(V,E)$ is an undirected simple graph with vertex set  $V=\{x_{1}x_{2}\cdots x_{n}\mid x_{i}=0~~ \emph {or}~~ 1, i=1,2,\ldots,n\}$. Two vertices $X=x_{1}x_{2}\cdots x_{n}$ and $Y=y_{1}y_{2}\cdots y_{n}$ are adjacent if $Y$ satisfies one of the following two conditions:
\begin{enumerate}
\item[(1)] $Y=x_{1}x_{2}\cdots x_{i-1}\overline x_{i}x_{i+1}\cdots x_{n}$ for some $1\le i \le n$;
\item[(2)] $Y=x_{1}x_{2}\cdots x_{k-1}\overline x_{k} \overline x_{k+1}\cdots \overline x_{n}$.
\end{enumerate}
\end{defi}
According to the above definition, we can see that $Q_{n, k}$ contains $Q_{n}$ as its subgraph. In fact,  $Q_{n, k}$ is a $(n+1)$-regular graph with $2^{n}$ vertices and $(n+1)2^{n-1}$ edges. Its special case of $k=1$ is the well-known folded hypercube denoted by $FQ_{n}$. As a variant of the hypercube, the n-dimensional folded hypercube $FQ_{n}$, proposed first by El-Amawy and Latifi \cite{El-Amawy}, is a graph obtained from the hypercube $Q_{n}$ by adding some edges, called a complementary edges, between  vertices $X=x_{1}x_{2}\cdots x_{n}$ and $\overline X=\overline x_{1}\overline x_{2}\cdots \overline x_{n}$. The enhanced hypercube is superior to the hypercube in many aspects. For example, the diameter of the enhanced hypercube is almost half of the hypercube. The hypercube is $n$-regular and $n$-connected, whereas the enhanced hypercube is $(n+1)$-regular and $(n+1)$-connected \cite{Liu-1}.

Let $G$ be a finite group, and let $S$ be a symmetric subset of $G$ such that $\mathbf{1}\not\in S$ and $S^{-1}=\{s^{-1}\mid s\in S\}=S$. The \emph{Cayley graph} on $G$ with respect to a symmetric subset $S$ of $G$, denoted by $Cay(G,S)$, is the undirected graph with vertex set $G$ and with an edge $\{g,h\}$ connecting $g$ and $h$ if $hg^{-1}\in S$, or equivalently $gh^{-1}\in S$. It is well known that $Cay(G,S)$ is connected if and only if $S$ generates $G$. Particularly, the group $Z_{2}^{n}=Z_{2}\times Z_{2}\times \cdots\times Z_{2}$ can be viewed as a vector space of dimension $n$ over the Galois filed $F_{2}$. Suppose that $\{e_{1},\ldots,e_{n}\}$ is the standard  basis of $Z_2^n$ and $\epsilon_{k}$ is the vector with first $k-1$ entries equal to $0$ and other entries equal to $1$ for some  $1\le k\le n-1$. Let $S=\{e_{1},\ldots, e_{n},\epsilon_{k}\}$ be the subset of $Z_2^n$. It is clear that  the so called enhanced hypercube $Q_{n, k}$ is just the Cayley graph $Cay(Z_2^n,S)$.

The concept of resistance distance of a graph $\Gamma=(V, E)$ was first introduced by Klein and Randi\'{c} \cite{Klein-1}. Let $\Gamma$ be a connected graph. The resistance distance between vertices $v_{i}$ and $v_{j}$ of $\Gamma$, denoted by $r_{ij}$, is defined  to be the effective resistance between the nodes $v_{i}$ and $v_{j}$ as computed with ohm's law when all the edges of $\Gamma$ are considered to be unit resistors. The traditional distance between vertices $v_{i}$ and $v_{j}$, denoted by $d_{ij}$, is the length of a shortest path connecting them. The Wiener index $W(\Gamma)$ was given by $W(\Gamma)=\sum\limits_{i<j}d_{ij}$ in \cite{Wiener}. As an analogue to the Wiener index, the sum $Kf(\Gamma)=\sum\limits_{i<j}r_{ij}$ was proposed in \cite{Klein-1}, later called the Kirchhoff index of $\Gamma$ in \cite{Bonchev}. Klein and Randi\'{c} \cite{Klein-1} proved that $r_{ij}\le d_{ij}$ and thus $Kf(\Gamma)\le W(\Gamma)$ with equality if and only if $\Gamma$ is a tree.

The Kirchhoff index has wide applications in physical interpretations, electric circuit, graph theory, and chemistry. For example, Gutman and Mohar \cite{Gutman} and Zhu, Klein et al. \cite{H.Y} proved that the Kirchhoff index of a connected graph $\Gamma$ with $n(n\ge2)$ vertices is the sum of reciprocal nonzero Laplacian eigenvalues of the graph multiplied by the number of the vertices. Like the Wiener index, the Kirchhoff index is a structure descriptor \cite{W.J}. The Kirchhoff index has been computed for cycles \cite{Klein-2, Lukovits}, complete graphs \cite{Lukovits}, geodetic graphs \cite{Palacios}, distance transitive graphs \cite{Palacios}, and so on. The Kirchhoff index of certain composite operations between two graphs was studied, such as product, lexicographic product \cite{H.Xu} and join, corona, cluster \cite{H.Zhang}.

In \cite{Liu-2}, we can see that the exact formula of the Kirchhoff index of the hypercubes networks $Q_{n}$ and related complex networks had been provided. Motivated by previous results, in this paper we present the formulae for the Kirchhoff index of the enhanced hypercube $Q_{n, k}$ and such formulae are brief for $Q_{n, 1}$ and $Q_{n, n-1}$. Moreover, we  prove that $Kf(Q_{n, k})$ is increased as $k$ increases for a given $n$. Additionally, the bounds of $Kf(Q_{n, k})$ and its Limit function is obtained, that is, $\lim\limits_{n\to\infty}\frac{Kf(Q_{n, k})}{\frac{2^{2n}}{n+1}}=1$.

\section{Preliminaries}\label{se-2}

Let $\Gamma$ be a simple graph with vertex set $V$ and edge set $E$. The \emph{adjacency matrix} $A(\Gamma)$ of $\Gamma$ is the $n\times n$ matrix with the $(i, j)$-entry equal to $1$ if vertices $v_{i}$ and $v_{j}$ are adjacent and $0$ otherwise. For $v_i\in V$, let $N(v_i)$ denote the set of neighbours of $v_i$, that is, $N(v_i)=\{v_j\in V\mid v_j\sim v_i\}$. The size of $N(v_i)$ is called the \emph{degree} of $v_i$, denoted by $d_i$. Let $D(\Gamma)$ be the diagonal matrix with $i$-th diagonal entry equal to $d_i$. The \emph{Laplacian matrix} of $\Gamma$ is defined by $L(\Gamma)=D(\Gamma)-A(\Gamma)$. The eigenvalues of $A(\Gamma)$ and $L(\Gamma)$ are called the \emph{adjacency eigenvalues} and \emph{Laplacian eigenvalues} of $\Gamma$, respectively. The multiset of adjacency (resp. Laplacian) eigenvalues together with their multiplicities is called the \emph{adjacency (resp. Laplacian) spectrum} of $\Gamma$, denoted by $\mathrm{Spec}_A(\Gamma)$ (resp. $\mathrm{Spec}_L(\Gamma)$).

In \cite{Babai}, Babai derived an expression for the spectrum of  the Cayley graph $Cay(G,S)$ in terms of irreducible characters of $G$. Here, we need only to know the eigenvalues of the Cayley graph of an abelian group.

\begin{lem}[\cite{Steinberg}]\label{lem-1}
Let $G=\{g_{1},g_{2},\ldots,g_{n}\}$ be an abelian group and $S\subseteq G$ is a symmetric set. Let $\chi_{1},\ldots,\chi_{n}$ be the irreducible characters of $G$ and Let $A$ be the adjacency matrix of the Cayley graph of $G$ with respect to $S$. Then the eigenvalues of the adjacency matrix $A$ are the real numbers
\[\lambda_{i}=\sum\limits_{s\in S}\chi_{i}(s)\]
where $1\le i\le n$.
\end{lem}

\begin{lem}[\cite{Steinberg}]\label{lem-2}
Let $G_{1},G_{2}$ be abelian groups and suppose that $\chi_{1},\ldots,\chi_{m}$ and $\varphi_{1},\ldots,\varphi_{n}$ are the irreducible representations of $G_{1}, G_{2}$, respectively. In particular, $m=|G_{1}|$ and $n=|G_{2}|$. Then the functions $\alpha_{ij}:G_{1}\times G_{2}\rightarrow \mathbb{C^{*}}$ with $1\le i\le m, 1\le j\le n$ given by
\[\alpha_{ij}(g_{1}, g_{2})=\chi_{i}(g_{1})\varphi_{j}(g_{2})\]
form a complete set of irreducible representations of $G_{1}\times G_{2}$.
\end{lem}

It is well known that $Z_2$ has two irreducible characters $\chi_0(a)=1(\forall a\in Z_2)$ and $\chi_1(a)=(-1)^{a}(\forall a\in Z_2)$. As a direct consequence of Lemma \ref{lem-2}, the irreducible characters $\chi_{i_{1},i_{2},\ldots,i_{n}}$ of $Z_{2}^{n}$ are given by
\begin{equation}\label{ch-eq-1}\chi_{i_{1},\ldots,i_{n}}((a_{1},\ldots,a_{n}))=(-1)^{i_{1}a_{1}+\cdots+i_{n}a_{n}}\end{equation}
where $(a_{1},\ldots,a_{n})\in Z_{2}^{n}$ and $i_{j}\in\{0,1\}$ for $1\le j\le n$. From (\ref{ch-eq-1}), Lemma \ref{lem-1} and Lemma \ref{lem-2}, we get the eigenvalues of $Cay(Z_2^n,S)$.

\begin{thm}\label{thm-1}
Let $S$ be a subset of  $Z_{2}^{n}$. Then the adjacency eigenvalues of the Cayley graph $Cay(Z_{2}^{n}, S)$ are given by
\[\lambda_{i_{1},i_{2},\ldots,i_{n}}=\sum\limits_{(s_{1},s_{2},\ldots,s_{n})\in S}(-1)^{i_{1}s_{1}+\cdots+i_{n}s_{n}}\]
where $(i_{1},\ldots,i_{n})\in Z_{2}^{n}$.
\end{thm}

Gutman and Mohar \cite{Gutman} and Zhu, Klein et al. \cite{H.Y} obtained the Kirchhoff index of a graph in terms of Laplacian eigenvalues as follows:

\begin{lem}[\cite{Gutman, H.Y}]\label{lem-3}
Let $\Gamma$ be a connected graph with $n\ge 2$ vertices. Then
\[Kf(\Gamma)=n\sum \limits_{i=1}^{n-1}\frac{1}{\mu_{i}},\]
where $0=\mu_{0}<\mu_{1}\le\cdots\le\mu_{n-1}$ are the Laplacian eigenvalues of $\Gamma$.
\end{lem}

\begin{lem}[\cite{Godsil}]\label{lem-4}
Let $A$ be the adjacency matrix of a graph $\Gamma$, and let $\rho$ be its spectral radius. Then the following are equivalent:\\
(1) $\Gamma$ is bipartite.\\
(2) The spectrum of $A$ is symmetric about the origin, i.e., for any $\lambda$, the multiplicities of $\lambda$ and $-\lambda$ as eigenvalues of $A$ are the same.\\
(3) $-\rho$ is an eigenvalue of $A$.
\end{lem}

\section{The eigenvalues of enhanced hypercubes}\label{se-3}

In this section, we focus on the eigenvalues of enhanced hypercubes. Let $\{e_{1},\ldots,e_{n}\}$ be the standard basis of $Z_{2}^{n}$, and let $\epsilon_{k}=e_k+e_{k+1}+\cdots+e_n$ for some $1\le k\le n-1$. The $n$-dimensional enhanced hypercube $Q_{n, k}(1\le k\le n-1)$ is just the Cayley graph $Cay(Z_{2}^{n}, S)$, where $Z_{2}^{n}=Z_{2}\times Z_{2}\times\cdots\times Z_{2}$ and $S=\{e_{1},\ldots, e_{n},\epsilon_{k}\}$.

\begin{lem}\label{lem-5}
For  $1\le k\le n-1$, the eigenvalues of $Q_{n, k}$ are given by $\zeta_{t}=n-2t-1$ with multiplicity $\alpha_{t}=\sum\limits_{j=1}^{n}{n-k+1\choose 2j-1}{k-1\choose t-2j+1}$ for  $t=1,2,...,n$, and $\xi_{t}=n-2t+1$ with multiplicity $\beta_{t}=\sum\limits_{j=0}^{n}{n-k+1\choose 2j}{k-1\choose t-2j}$ for $t=0,1,2,...,n$.
\end{lem}
\begin{proof}
Since $Q_{n, k}\cong Cay(Z_2^n, S)$, we need only to calculate the eigenvalues of $Cay(Z_{2}^{n}, S)$. For $v=(v_{1},v_{2},\ldots,v_{n})^T\in Z_2^n$, according to (\ref{ch-eq-1}), the value of the irreducible character  $\chi_{v}$ at $a=(a_1,a_2,...,a_n)^T\in Z_2^n$ is $\chi_{v}(a)=(-1)^{v_{1}a_{1}+v_{2}a_{2}+\cdots+v_{n}a_{n}}$. Thus
\begin{equation}\label{ch-eq-2}
\left\{\begin{array}{ll}\chi_{v}(e_i)=(-1)^{v_{i}}\\
\chi_{v}(\epsilon_k)=(-1)^{v_{k}+v_{k+1}+\cdots+v_n}
\end{array}\right.
\end{equation}
From (\ref{ch-eq-2}) and Theorem \ref{thm-1}, the eigenvalue corresponding to $v=(v_{1},v_{2},\ldots,v_{n})^T\in Z_2^n$ is
\begin{equation}\label{ch-eq-3}\lambda_v=\sum_{s\in S}\chi_v(s)=(-1)^{v_{1}}+(-1)^{v_{2}}+\cdots+(-1)^{v_{n}}+(-1)^{v_{k}+v_{k+1}+\cdots+v_n}
\end{equation}
Given $0\le r\le t\le n$, let
$\Lambda(t, r)=\{v\in Z_2^n\mid v_{1}+v_{2}+\cdots+v_{n}=t, v_{k}+v_{k+1}+\cdots+v_n=r\}$. It is easy to see  that $|\Lambda(t, r)|={n-k+1\choose r}{k-1\choose t-r}$. From (\ref{ch-eq-3}), for any $v\in\Lambda(t, r)$, we have
\begin{equation}\label{ch-eq-4}\lambda_v=\sum_{s\in S}\chi_v(s)=n-2t+(-1)^r\end{equation}
Clearly, $Z_2^n=\cup_{0\le r\le t\le n}\Lambda(t, r)$ is a partition of $Z_2^n$. By taking $r=2j-1$ for $1\le j\le \left\lfloor \frac{t+1}{2}\right\rfloor$, from (\ref{ch-eq-4}) we get the eigenvalue $\zeta_{t}=n-2t-1$, where $t=1,2,...,n$,  with multiplicity $|\Lambda(t, r)|=\sum\limits_{j=1}^{\left\lfloor \frac{t+1}{2}\right\rfloor}{n-k+1\choose 2j-1}{k-1\choose t-2j+1}$ which equals $\alpha_{t}=\sum\limits_{j=1}^{n}{n-k+1\choose 2j-1}{k-1\choose t-2j+1}$;  by taking $r=2j$ for $0\le j\le \left\lfloor \frac{t}{2} \right\rfloor$, we get the eigenvalue $\xi_{t}=n-2t+1$, where $t=0,1,2,...,n$, with multiplicity $|\Lambda(t, r)|=\sum\limits_{j=0}^{\left\lfloor \frac{t}{2} \right\rfloor}{n-k+1\choose 2j}{k-1\choose t-2j}$ which equals $\beta_{t}=\sum\limits_{j=0}^{n}{n-k+1\choose 2j}{k-1\choose t-2j}$.

This completes the proof.
\end{proof}

Note that some of the two kind of eigenvalues described in Lemma \ref{lem-5} may be equal, in fact, $\zeta_t=\xi_{t+1}$ for $t=1,2,...,n-1$. Thus  we can give the spectrum of $Q_{n, k}$ from Lemma \ref{lem-5}.

\begin{thm}\label{thm-2}
Let $1\le k\le n-1$, the spectrum of $Q_{n, k}$ is given by \[\mathrm{Spec}_{A}(Q_{n, k})=\{n+1, [n-1]^{k-1}, [n-2t-1]^{\gamma_t},[-n-1]^{\gamma_n} \mid t=1,2,...,n-1\},\] where $\gamma_t=\sum\limits_{j=0}^{n}{n-k+2\choose 2j}{k-1\choose t+1-2j}$ for $1\le t\le n-1$ and $\gamma_n=\sum\limits_{j=1}^{n}{n-k+1\choose 2j-1}{k-1\choose n-2j+1}$.
\end{thm}
\begin{proof}
For convenience, denote by $\lambda_0>\lambda_1>\lambda_{2}>\cdots$ the distinct eigenvalues of $Q_{n, k}$. By Lemma \ref{lem-5}, we know that $\lambda_0=\xi_{0}=n+1$ with multiplicity $\beta_{0}=\sum\limits_{j=0}^{n}{n-k+1\choose 2j}{k-1\choose 0-2j}=1$. $\lambda_1=\xi_{1}=n-1$ with multiplicity $\beta_{1}=\sum\limits_{j=0}^{n}{n-k+1\choose 2j}{k-1\choose 1-2j}={k-1\choose 1}=k-1$. In addition, we see that $\lambda_{1+t}=\zeta_t=\xi_{t+1}=n-2t-1$ for $t=1,2,...,n-1$, so the multiplicity of the eigenvalue $n-2t-1$ is $\alpha_{t}+\beta_{t+1}$, where $\gamma_t=\alpha_{t}+\beta_{t+1}$ can be simplified as
\[\begin{array}{lll}
\gamma_{t}&=&\alpha_{t}+\beta_{t+1}=\sum\limits_{j=1}^{n}{n-k+1\choose 2j-1}{k-1\choose t-2j+1}+\sum\limits_{j=0}^{n}{n-k+1\choose 2j}{k-1\choose t-2j+1}\\
&=&\sum\limits_{j=1}^{n}\left({n-k+1\choose 2j-1}+{n-k+1\choose 2j}\right){k-1\choose t-2j+1}+{k-1\choose t+1}\\
&=&\sum\limits_{j=0}^{n}{n-k+2\choose 2j}{k-1\choose t+1-2j}.\\
\end{array}\]
At last, $\lambda_{n+1}=\zeta_n=-n-1$ with multiplicity $\gamma_n=\alpha_{n}=\sum\limits_{j=1}^{n}{n-k+1\choose 2j-1}{k-1\choose n-2j+1}$.

This completes the proof.
\end{proof}
Since the enhanced hypercube $Q_{n, k}$ is regular, the Laplacian spectrum of $Q_{n, k}$ can be reduced by its $A$-spectrum. From Theorem \ref{thm-2}, we get the following result.

\begin{thm}\label{thm-3}
Let $1\le k\le n-1$, the Laplacian spectrum of $Q_{n, k}$ is given by
\[\mathrm{Spec}_{L}(Q_{n, k})=\{0,[2]^{k-1}, [2t+2]^{\gamma_{t}},[2n+2]^{\gamma_n} \mid t=1,2,...,n-1\},\]
where $\gamma_t=\sum\limits_{j=0}^{n}{n-k+2\choose 2j}{k-1\choose t+1-2j}$ for $1\le t\le n-1$ and $\gamma_n=\sum\limits_{j=1}^{n}{n-k+1\choose 2j-1}{k-1\choose n-2j+1}$.

\end{thm}

\begin{proof}
Since $Q_{n, k}$ is $n+1$ regular, we have $L(Q_{n, k})=(n+1)I-A(Q_{n, k})$. Thus every eigenvector of $A(Q_{n, k})$ with eigenvalue $\lambda$ is an eigenvector of $L(Q_{n, k})$  with eigenvalue $(n+1)-\lambda$. By Theorem \ref{thm-2}, the result yields immediately.
\end{proof}

In general, the summations of $\alpha_{t}$ and $\beta_{t}$ can not be simply calculated. However, we can find the simple expression of the spectrum for some special $k$. By Theorem \ref{thm-2} and Theorem \ref{thm-3}, we give the $A$-spectrum and Laplacian spectrum of $Q_{n, k}$ for $k=1$.

\begin{cor}\label{cor-1}
For $k=1$ and $n\ge2$, the $A$-spectrum and Laplacian spectrum of $Q_{n, 1}$ is given by
\begin{itemize}
\item[(1)] $ \mathrm{Spec}_{A}(Q_{n,1})=\left\{\begin{array}{ll}
\{n+1, [n-3]^{\binom{n+1}{2}},\ldots,[-n+3]^{\binom{n+1}{n-1}},-n-1\} &\mbox{if $n$ is odd}\\
\{n+1, [n-3]^{\binom{n+1}{2}},\ldots,[-n+5]^{\binom{n+1}{n-2}},[-n+1]^{n+1}\} &\mbox{if $n$ is even}.
\end{array}
\right. $
\item[(2)] $ \mathrm{Spec}_{L}(Q_{n,1})=\left\{\begin{array}{ll}
\{0, [4]^{\binom{n+1}{2}},[8]^{\binom{n+1}{4}},\ldots,[2n-2]^{\binom{n+1}{n-1}},2n+2\} &\mbox{if $n$ is odd}\\
\{0, [4]^{\binom{n+1}{2}},[8]^{\binom{n+1}{4}},\ldots,[2n-4]^{\binom{n+1}{n-2}},[2n]^{n+1}\} &\mbox{if $n$ is even}.
\end{array}
\right.$
\end{itemize}
\end{cor}
\begin{remark}
The folded hypercube $FQ_{n}$ introduced in \cite{Chen} is just the enhanced hypercube $Q_{n, k}$ for $k=1$, i.e., $FQ_{n}=Q_{n, 1}$. Thus Corollary \ref{cor-1} gives the $A$-spectrum and Laplacian spectrum of the folded hypercube $FQ_{n}$.
\end{remark}

We give the following two Corollaries at the last of this section, which can be regarded as an application of Theorem \ref{thm-2} .
Denote  the adjacency eigenvalues of $\Gamma$ by $\lambda_{1}\ge \lambda_{2}\ge\cdots\ge\lambda_{n}$, and $\lambda_{1}-\lambda_{2}$ is the so called \emph{spectral gap} of the adjacency matrix, which is also an important spectral parameter related to the expander property \cite{Shlomo}.
\begin{cor}\label{cor-2}
For $2\le k\le n-1$, the second largest eigenvalue of $Q_{n,k}$ is $n-1$, and so the adjacency spectral gap of $Q_{n, k}$ is 2. Especially, the second largest eigenvalue of $Q_{n,1}$ is $n-3$, the corresponding adjacency spectral gap of $Q_{n, 1}$ is 4.
\end{cor}
\begin{proof}
From Theorem \ref{thm-2}, we known that $n+1$ is the first largest eigenvalue of $Q_{n, k}$. Moreover, for $k\ge2$, the second largest eigenvalue is $n-1$, for $k=1$, the second largest eigenvalue of $Q_{n, 1}$ is $n-3$.
\end{proof}
H. M. Liu in \cite{Liu-3} characterized the bipartite graphs among enhanced hypercubes $Q_{n, k}$. Here we give it another simple  proof by the spectrum of $Q_{n, k}$.
\begin{cor}\label{cor-3}
$Q_{n, k}$ is a bipartite graph if and only if $n$ and $k$ have the same parity for $1\le k\le n-1$.
\end{cor}
\begin{proof}
By Lemma \ref{lem-3}, $Q_{n, k}$ is a bipartite graph if and only if $-n-1$ is an eigenvalue with multiplicity one, by theorem \ref{thm-1}, if and only if $\gamma_{n}=\sum\limits_{j=1}^{n}{n-k+1\choose 2j-1}{k-1\choose n-2j+1}=1$.  Note that the valid term in the summation of $\gamma_n$ must satisfy $n-k+1\ge2j-1$ and $k-1\ge n-2j+1$. Hence $n-k+1=2j-1$, that is, $j=\frac{n-k+2}{2}$. So we have $\gamma_{n}=\sum\limits_{j=1}^{n}{n-k+1\choose 2j-1}{k-1\choose n-2j+1}={n-k+1\choose n-k+1}{k-1\choose k-1}=1$ if  $n$ and $k$ have the same parity and $\gamma_{n}=0$ otherwise. It immediately follows our result.
\end{proof}

\section{The Kirchhoff index of enhanced hypercubes}\label{se-4}

In this section, we focus on our main results. First we give the formula of the Kirchhoff index of the enhanced hypercube $Q_{n, k}$ where $1\le k\le n-1$. Next we will show the monotonicity of $Kf(Q_{n, k})$. At last, we get a  limiting function for $Kf(Q_{n, k})$.

\begin{thm}\label{thm-4}
Let $1\le k\le n-1$. The Kirchhoff index of the enhanced hypercube $Q_{n, k}$ is given by
\[Kf(Q_{n, k})=\left\{\begin{array}{ll}
2^{n-1}\sum\limits_{t=0}^{n}\sum\limits_{j=0}^{n}\frac{1}{t+1}{n-k+2\choose 2j}{k-1\choose t+1-2j} &\mbox{if $n\equiv k~(\mathrm{mod}~2)$}\\
2^{n-1}\sum\limits_{t=0}^{n-1}\sum\limits_{j=0}^{n}\frac{1}{t+1}{n-k+2\choose 2j}{k-1\choose t+1-2j} &\mbox{if $n\not\equiv k~(\mathrm{mod}~2)$}.
\end{array}\right.\]
\end{thm}
\begin{proof}
We denote the Laplacian eigenvalues of $Q_{n, k}$ by $\mu_{i}$ for $i=0,1,...,2^n-1$. By Lemma \ref{lem-3} and Theorem \ref{thm-3} we have
\[\begin{array}{lll}
Kf(Q_{n, k})&=&\displaystyle 2^{n}\sum\limits_{i=1}^{2^{n}-1}\frac{1}{\mu_{i}}=2^{n}\sum\limits_{\mu\in \mathrm{Spec}_{L}(Q_{n, k})\backslash 0 }\frac{1}{\mu}\\
&=&2^{n}\left(\frac{k-1}{2}+\frac{\gamma_{n}}{2n+2}+\sum\limits_{t=1}^{n-1}\frac{\gamma_{t}}{2t+2}\right)\\
&=&2^{n-1}\left(k-1+\frac{\gamma_n}{n+1}+\sum\limits_{t=1}^{n-1}\frac{\gamma_{t}}{t+1}\right).
\end{array}\]
By Corollary \ref{cor-3}, we known that  $\gamma_{n}=1$ if  $n$ and $k$ have the same parity, and $\gamma_{n}=0$ otherwise. By Theorem \ref{thm-2}, we have $\gamma_{t}=\sum\limits_{j=0}^{n}{n-k+2\choose 2j}{k-1\choose t+1-2j}$ for $t=1,2,...,n-1$. Therefore, we have
\begin{equation}\label{Q-eq-1}Kf(Q_{n, k})=\left\{\begin{array}{ll}
2^{n-1}\left(k-1+\frac{1}{n+1}+\sum\limits_{t=1}^{n-1}\sum\limits_{j=0}^{n}\frac{1}{t+1}{n-k+2\choose 2j}{k-1\choose t+1-2j}\right)&\mbox{if $n\equiv k~(\mathrm{mod}~2)$}\\
2^{n-1}\left(k-1+\sum\limits_{t=1}^{n-1}\sum\limits_{j=0}^{n}\frac{1}{t+1}{n-k+2\choose 2j}{k-1\choose t+1-2j}\right) &\mbox{if $n\not\equiv k~(\mathrm{mod}~2)$}.
\end{array}\right.\end{equation}
It is easy to verify that  $\sum\limits_{j=0}^{n}\frac{1}{t+1}{n-k+2\choose 2j}{k-1\choose t+1-2j}=k-1$ if $t=0$. On the other hand, if $t=n$, we have $\sum\limits_{j=0}^{n}\frac{1}{t+1}{n-k+2\choose 2j}{k-1\choose t+1-2j}=\frac{1}{n+1}\sum\limits_{j=0}^{n}{n-k+2\choose 2j}{k-1\choose n+1-2j}=\frac{1}{n+1}$ for $n\equiv k~(\mathrm{mod}~2)$. Thus (\ref{Q-eq-1}) becomes
\[Kf(Q_{n, k})=\left\{\begin{array}{ll}
2^{n-1}\sum\limits_{t=0}^{n}\sum\limits_{j=0}^{n}\frac{1}{t+1}{n-k+2\choose 2j}{k-1\choose t+1-2j} &\mbox{if $n\equiv k~(\mathrm{mod}~2)$}\\
2^{n-1}\sum\limits_{t=0}^{n-1}\sum\limits_{j=0}^{n}\frac{1}{t+1}{n-k+2\choose 2j}{k-1\choose t+1-2j} &\mbox{if $n\not\equiv k~(\mathrm{mod}~2)$}.
\end{array}\right.\]

This completes the proof.
\end{proof}

Theorem \ref{thm-4} gives an exact formula of the Kirchhoff index of $Q_{n, k}$. However, its representation has some complicated. Now we give a lemma that will be used to simplify the Kirchhoff index of $Q_{n, k}$ for some $k$.

\begin{lem}\label{lem-6}
For a positive integer $n$, we have $\sum\limits_{i=1}^{n}\frac{1}{2i}{n+1\choose 2i}=\sum\limits_{s=1}^{n}\frac{2^{s}-1}{s+1}$.
\end{lem}
\begin{proof}
By the fundamental identity of combination ${n-1\choose 2i-1}+{n-1\choose 2i}={n\choose 2i}$ and $\frac{1}{2i}{n-1\choose 2i-1}=\frac{1}{n}{n\choose 2i}$, we get
\[\begin{array}{lll}
\sum\limits_{i=1}^{n}\frac{1}{2i}{n+1\choose 2i}&=&\sum\limits_{i=1}^{n}\left(\frac{1}{2i}{n\choose2i-1}+\frac{1}{2i}{n\choose 2i}\right)=\sum\limits_{i=1}^{n}\left(\frac{1}{n+1}{n+1\choose2i}+\frac{1}{2i}{n\choose 2i}\right)\\
&=&\sum\limits_{i=1}^{n}\left(\frac{1}{n+1}{n+1\choose2i}+\frac{1}{2i}{n-1\choose 2i-1}+\frac{1}{2i}{n-1\choose2i}\right)\\
&=&\sum\limits_{i=1}^{n}\left(\frac{1}{n+1}{n+1\choose2i}+\frac{1}{n}{n\choose 2i}+\frac{1}{2i}{n-2\choose2i-1}+\frac{1}{2i}{n-2\choose2i}\right)\\
&=&\sum\limits_{i=1}^{n}\left[\frac{1}{n+1}{n+1\choose2i}+\frac{1}{n}{n\choose 2i}+\cdots+\frac{1}{n-(n-2i)}{n-(n-2i)\choose2i}+\frac{1}{2i}{n-(n-2i)-1\choose2i}\right]\\
&=&\sum\limits_{i=1}^{n}\sum\limits_{r=0}^{n}\frac{1}{n+1-r}{n+1-r\choose2i}=\sum\limits_{r=0}^{n}\frac{1}{n+1-r}\sum\limits_{i=1}^{n}{n+1-r\choose2i}\\
&=&\sum\limits_{r=0}^{n}\frac{1}{n+1-r}(2^{n-r}-1)=\sum\limits_{s=1}^{n}\frac{2^{s}-1}{s+1}\\
\end{array}\]
This completes the proof.
\end{proof}

By Lemma \ref{lem-6}, we arrive at the following result, which is a special case of Theorem \ref{thm-4}. In fact, the formula of Kirchhoff index of the folded hypercube $FQ_n$ have been presented in \cite{Liu-4}. Corollary \ref{cor-4} simplifies the formula.

\begin{cor}\label{cor-4}
The Kirchhoff index of the enhanced hypercube $Q_{n, 1}$ is given by
\[\begin{array}{ll}Kf(Q_{n, 1})=2^{n-1}\sum\limits_{t=1}^{n}\frac{2^{t}-1}{t+1}\end{array}.\]
\end{cor}

\begin{proof}
By Theorem \ref{thm-4}, the Kirchhoff index of $Q_{n, 1}$ can be presented by
\[Kf(Q_{n, 1})=\left\{\begin{array}{ll}
2^{n-1}\sum\limits_{t=0}^{n}\sum\limits_{j=0}^{n}\frac{1}{t+1}{n+1\choose 2j}{0\choose t+1-2j} &\mbox{if $n$ is odd}\\
2^{n-1}\sum\limits_{t=0}^{n-1}\sum\limits_{j=0}^{n}\frac{1}{t+1}{n+1\choose 2j}{0\choose t+1-2j}&\mbox{if $n$ is even}.
\end{array}\right.\]
It is easy to check that
\[\begin{array}{lll}
\sum\limits_{t=0}^{n}\sum\limits_{j=0}^{n}\frac{1}{t+1}{n+1\choose 2j}{0\choose t-2j+1}=\sum\limits_{j=1}^{n}\frac{1}{2j}{n+1\choose 2j}{0\choose 0}=\sum\limits_{j=1}^{n}\frac{1}{2j}{n+1\choose 2j}
\end{array}\]
and
\[\begin{array}{lll}
\sum\limits_{t=0}^{n-1}\sum\limits_{j=0}^{n}\frac{1}{t+1}{n+1\choose 2j}{0\choose t-2j+1}=\sum\limits_{j=1}^{n}\frac{1}{2j}{n+1\choose 2j}{0\choose 0}=\sum\limits_{j=1}^{n}\frac{1}{2j}{n+1\choose 2j}
\end{array}.\]
Combined with the above two situations, we get
\[\begin{array}{lll}
Kf(Q_{n, 1})&=&2^{n-1}\sum\limits_{j=1}^{n}\frac{1}{2j}{n+1\choose2j}
\end{array}.\]
By lemma \ref{lem-6}, we have
\[\begin{array}{ll}
Kf(Q_{n, 1})=2^{n-1}\sum\limits_{j=1}^{n}\frac{1}{2j}{n+1\choose2j}=2^{n-1}\sum\limits_{t=1}^{n}\frac{2^{t}-1}{t+1}.\\
\end{array}\]

This completes the proof.
\end{proof}
\begin{cor}\label{cor-5}
The Kirchhoff index of the enhanced hypercube $Q_{n, n-1}$ is given by
\[\begin{array}{ll}
Kf(Q_{n, n-1})=2^{n-1}\left(\sum\limits_{t=1}^{n-2}\frac{2^{t}-1}{t}+3\frac{(n-2)2^{n-1}+1}{n(n-1)}\right)
\end{array}.\]
\end{cor}
\begin{proof}
Let $k=n-1$, then $n$ and $k$ have the different parity. By Theorem \ref{thm-4}, we have
\[\begin{array}{lll}
Kf(Q_{n, n-1})&=&2^{n-1}\sum\limits_{t=0}^{n-1}\sum\limits_{j=0}^{n}\frac{1}{t+1}{3\choose2j}{n-2\choose t+1-2j}\\
&=&2^{n-1}\sum\limits_{t=0}^{n-1}\frac{1}{t+1}\left({3\choose 0}{n-2\choose t+1}+{3\choose 2}{n-2\choose t-1}\right)\\
&=&2^{n-1}\left(\sum\limits_{t=0}^{n-1}\frac{1}{t+1}{n-2\choose t+1}+3\sum\limits_{t=0}^{n-1}\frac{1}{t+1}{n-2\choose t-1}\right)\\
\end{array}\]
By Lemma \ref{lem-6}, it follows that
\[\begin{array}{ll}
 \sum\limits_{t=0}^{n-1}\frac{1}{t+1}{n-2\choose t+1}=\sum\limits_{r=1}^{n}\frac{1}{r}{n-2\choose r}=\sum\limits_{s=1}^{n-2}\frac{2^{s}-1}{s}
 \end{array}.\]
Besides,
\[\begin{array}{lll}
\sum\limits_{t=0}^{n-1}\frac{1}{t+1}{n-2\choose t-1}&=&\sum\limits_{t=0}^{n-1}\frac{1}{t+1}\left({n-1\choose t}-{n-2\choose t}\right)=\sum\limits_{t=0}^{n-1}\frac{1}{t+1}{n-1\choose t}-\sum\limits_{t=0}^{n-1}\frac{1}{t+1}{n-2\choose t}\\
&=&\frac{1}{n}\sum\limits_{t=0}^{n-1}{n\choose t+1}-\frac{1}{n-1}\sum\limits_{t=0}^{n-1}{n-1\choose t+1}=\frac{1}{n(n-1)}[(n-2)2^{n-1}+1]
\end{array}\]
Hence we obtain that
\[\begin{array}{ll}
Kf(Q_{n,n-1})=2^{n-1}\left(\sum\limits_{t=1}^{n-2}\frac{2^{t}-1}{t}+3\frac{(n-2)2^{n-1}+1}{n(n-1)}\right)
\end{array}.\]

This completes the proof.
\end{proof}

\begin{table}[h]{\footnotesize
\caption{\label{tab-1}{Kirchhoff index of $Q_{n, k}$ ($1\le k\le n-1$).}}
\resizebox{\textwidth}{26mm}{
\begin{tabular}{c|ccccccccccc}
\hline
\diagbox{n}{$Kf(Q_{n, k})$}{k}&$1$&$2$&$3$&$4$&$5$&$6$&$7$&$8$&$9$\\
\hline
$2$&3&--&--&--&--&--&--&--&--\\
\hline
$3$&13&14&--&--&--&--&--&--&--\\
\hline
$4$&50&51.6&54&--&--&--&--&--&--\\
\hline
$5$&182.7&185.3&189.1&194.9&--&--&--&--&--\\
\hline
$6$&653.3&657.9&664&672.8&687.5&--&--&--&--\\
\hline
$7$&2322.7&2330.7&2341.0&2355.0&2376.3&2413.6&--&--&--\\
\hline
$8$&8272&8286.2&8304&8327.4&8360.4&8412.4&8509.4&--&--\\
\hline
$9$&29626&29651&29682&29722&29776&29854&29984&30242&--\\
\hline
$10$&106870&106910&106970&107040&107130&107250&107440&107770&108480\\
\hline
\end{tabular}}}
\end{table}

We calculate the Kirchhoff index of $Q_{n, k}$ for some $n$ and $k$ in Tabla \ref{tab-1}, from which we see that the Kirchhoff index of $Q(n, k)$ is increased with the increase of $k$. Fortunately, it is true in general. To prove this property, we lead-in the following lemma.
\begin{lem}\label{lem-7}
 Let $1\le k\le n-1$. We have
$\sum\limits_{t=0}^{n}\sum\limits_{j=0}^{n}\frac{1}{t+1}\left[{n-k\choose 2j}{k-1\choose t-2j-1}-{n-k\choose 2j-1}{k-1\choose t-2j}\right]\ge0$.
\end{lem}
\begin{proof}
Since $1\le k\le n-1$, we have $1\le n-k\le n-1$ and $0\le k-1\le n-2$. It is clear that
$F(n,k)=\sum\limits_{t=0}^{n}\sum\limits_{j=0}^{n}\frac{1}{t+1}\left[{n-k\choose 2j}{k-1\choose t-2j-1}-{n-k\choose 2j-1}{k-1\choose t-2j}\right]=\sum\limits_{t=0}^{n}\frac{1}{t+1}\sum\limits_{r=0}^{n}(-1)^{r}{n-k\choose r}{k-1\choose (t-1)-r}$, from which
 we define a function $F_{n, k}(x)=\sum\limits_{t=0}^{n}\frac{1}{t+1}\sum\limits_{r=0}^{n}(-1)^{r}{n-k\choose r}{k-1\choose (t-1)-r}x^{t}$. Obviously, $F_{n, k}(1)=F(n,k)$.
Now we introduce another function,
\[\begin{array}{lll}
G(x)&=&x(1-x)^{n-k}(1+x)^{k-1}=x\sum\limits_{i=0}^{n-k}(-1)^{i}{n-k\choose i}x^{i}\sum\limits_{j=0}^{k-1}{k-1\choose j}x^{j}\\
&=&x\sum\limits_{i=0}^{n-k}\sum\limits_{j=0}^{k-1}(-1)^{i}{n-k\choose i}{k-1\choose j}x^{i+j}=x\sum\limits_{t=1}^{n}\sum\limits_{r=0}^{n}(-1)^{r}{n-k\choose r}{k-1\choose(t-1)-r}x^{t-1}\\
&=&\sum\limits_{t=1}^{n}\sum\limits_{r=0}^{n}(-1)^{r}{n-k\choose r}{k-1\choose(t-1)-r}x^{t}=\sum\limits_{t=0}^{n}\sum\limits_{r=0}^{n}(-1)^{r}{n-k\choose r}{k-1\choose(t-1)-r}x^{t}.
\end{array}\]
Then we have
\[\begin{array}{ll}
H(x)=\int_{0}^{x}G(u)\,du&=\int_{0}^{x}u(1-u)^{n-k}(1+u)^{k-1}\,du\\
&=\int_{0}^{x}\sum\limits_{t=0}^{n}\sum\limits_{r=0}^{n}(-1)^{r}{n-k\choose r}{k-1\choose(t-1)-r}u^{t}\,du\\
&=\sum\limits_{t=0}^{n}\sum\limits_{r=0}^{n}(-1)^{r}{n-k\choose r}{k-1\choose(t-1)-r}\int_{0}^{x}u^{t}\,du\\
&=\sum\limits_{t=0}^{n}\frac{1}{t+1}\sum\limits_{r=0}^{n}(-1)^{r}{n-k\choose r}{k-1\choose (t-1)-r}x^{t+1}.
\end{array}\]
Hence, $H(1)=F_{n, k}(1)=F(n,k)$. On the other hand,
\[\begin{array}{lll}
H(1)=\int_{0}^{1}G(u)\,du=\int_{0}^{1}u(1-u)^{n-k}(1+u)^{k-1}\,du.
\end{array}\]
Since $G(u)=u(1-u)^{n-k}(1+u)^{k-1}\ge0$ while $0\le u\le1$, we have
\[\begin{array}{lll}
F(n, k)=\int_{0}^{1}G(u)\,du=\int_{0}^{1}u(1-u)^{n-k}(1+u)^{k-1}\,du\ge0.
\end{array}\]
This completes the proof.
\end{proof}
\begin{thm}\label{thm-5}
Let $1\le k\le n-1$. For a given $n$, the Kirchhoff index of the enhanced hypercube $Q_{n, k}$ is increased with the increase of $k$ .
\end{thm}
\begin{proof}
 Denote by $\triangledown_{k}=Kf(Q_{n, k+1})-Kf(Q_{n, k})$ for $1\le k<n-1$. We will show that $\triangledown_{k}>0$ by induction on $k$.
By Theorem \ref{thm-4}, if $n\not\equiv k~(\mathrm{mod}~2)$, we have
\begin{equation}\label{delta-1}\begin{array}{ll}
\triangledown_{k}&=2^{n-1}\sum\limits_{t=0}^{n}\sum\limits_{j=0}^{n}\frac{1}{t+1}{n-k+1\choose 2j}{k\choose t+1-2j}-2^{n-1}\sum\limits_{t=0}^{n-1}\sum\limits_{j=0}^{n}\frac{1}{t+1}{n-k+2\choose 2j}{k-1\choose t+1-2j}\\
&=2^{n-1}\sum\limits_{t=0}^{n}\sum\limits_{j=0}^{n}\frac{1}{t+1}{n-k+1\choose 2j}{k\choose t+1-2j}-2^{n-1}\sum\limits_{t=0}^{n}\sum\limits_{j=0}^{n}\frac{1}{t+1}{n-k+2\choose 2j}{k-1\choose t+1-2j}
\end{array}
\end{equation}
The second equality holds because $\sum\limits_{j=0}^{n}\frac{1}{t+1}{n-k+2\choose 2j}{k-1\choose t+1-2j}=0$ if $t=n$( in fact, by $n-k+2\ge 2j$ and $k-1\ge  n+1-2j$, we obtain that $n-k=2j-2$, which contradict to $n\not\equiv k~(\mathrm{mod}~2)$). Thus (\ref{delta-1}) can be simplified as
\begin{equation}\label{delta-2}\begin{array}{ll}
\triangledown_{k}&=2^{n-1}\sum\limits_{t=0}^{n}\sum\limits_{j=0}^{n}\frac{1}{t+1}\left[{n-k+1\choose 2j}{k\choose t-2j+1}-{n-k+2\choose 2j}{k-1\choose t-2j+1}\right]\\
&=2^{n-1}\sum\limits_{t=0}^{n}\sum\limits_{j=0}^{n}\frac{1}{t+1}\left[{n-k+1\choose 2j}\left({k-1\choose t+1-2j}+{k-1\choose t-2j}\right)-\left({n-k+1\choose 2j-1}+{n-k+1\choose 2j}\right){k-1\choose t+1-2j}\right]\\
&=2^{n-1}\sum\limits_{t=0}^{n}\sum\limits_{j=0}^{n}\frac{1}{t+1}\left[{n-k+1\choose 2j}{k-1\choose t-2j}-{n-k+1\choose 2j-1}{k-1\choose t-2j+1}\right].
\end{array}
\end{equation}
If $n\equiv k~(\mathrm{mod}~2)$, we have
\begin{equation}\label{delta-3}\begin{array}{ll}
\triangledown_{k}&=2^{n-1}\sum\limits_{t=0}^{n-1}\sum\limits_{j=0}^{n}\frac{1}{t+1}{n-k+1\choose 2j}{k\choose t+1-2j}-2^{n-1}\sum\limits_{t=0}^{n}\sum\limits_{j=0}^{n}\frac{1}{t+1}{n-k+2\choose 2j}{k-1\choose t+1-2j}\\
&=2^{n-1}\sum\limits_{t=0}^{n}\sum\limits_{j=0}^{n}\frac{1}{t+1}\left[{n-k+1\choose 2j}{k\choose t-2j+1}-{n-k+2\choose 2j}{k-1\choose t-2j+1}\right]\\
&=2^{n-1}\sum\limits_{t=0}^{n}\sum\limits_{j=0}^{n}\frac{1}{t+1}\left[{n-k+1\choose 2j}\left({k-1\choose t+1-2j}+{k-1\choose t-2j}\right)-\left({n-k+1\choose 2j-1}+{n-k+1\choose 2j}\right){k-1\choose t+1-2j}\right]\\
&=2^{n-1}\sum\limits_{t=0}^{n}\sum\limits_{j=0}^{n}\frac{1}{t+1}\left[{n-k+1\choose 2j}{k-1\choose t-2j}-{n-k+1\choose 2j-1}{k-1\choose t-2j+1}\right].
\end{array}
\end{equation}
The second equality holds because $\sum\limits_{j=0}^{n}\frac{1}{t+1}{n-k+1\choose 2j}{k\choose t-2j+1}=0$ if $t=n$ ( in fact, by $n-k+1\ge 2j$ and $k\ge n-2j+1$, we can imply that
$n-k=2j-1$, which contradict to $n\equiv k~(\mathrm{mod}~2)$).

The last representations of $\triangledown_{k}$ in  (\ref{delta-2}) and (\ref{delta-3}) are the same. Hence we need not to distinguish the above  two cases in what follows.

For $k=1$, from (\ref{delta-2}) we have
\begin{equation}\label{qq-1}\begin{array}{ll}
\triangledown_{1}&=2^{n-1}\sum\limits_{t=0}^{n}\sum\limits_{j=0}^{n}\frac{1}{t+1}\left[{n\choose 2j}{0\choose t-2j}-{n\choose 2j-1}{0\choose t-2j+1}\right]\\
&=2^{n-1}\left(\sum\limits_{t=0}^{n}\sum\limits_{j=0}^{n}\frac{1}{t+1}{n\choose 2j}{0\choose t-2j}-\sum\limits_{t=0}^{n}\sum\limits_{j=0}^{n}\frac{1}{t+1}{n\choose 2j-1}{0\choose t-2j+1} \right)\\
&=2^{n-1}\left(\sum\limits_{j=0}^{n}\frac{1}{2j+1}{n\choose 2j}-\sum\limits_{j=1}^{n}\frac{1}{2j}{n\choose 2j-1}\right)\\
&=2^{n-1}\left(\frac{1}{n+1}\sum\limits_{j=0}^{n}{n+1\choose 2j+1}-\frac{1}{n+1}\sum\limits_{j=1}^{n}{n+1\choose 2j}\right)=\frac{2^{n-1}}{n+1}>0.
\end{array}
\end{equation}
Now we assume that $\triangledown_{k}=2^{n-1}\sum\limits_{t=0}^{n}\sum\limits_{j=0}^{n}\frac{1}{t+1}\left[{n-k+1\choose 2j}{k-1\choose t-2j}-{n-k+1\choose 2j-1}{k-1\choose t-2j+1}\right]>0$ holds for  $1\le k < n-2$. Next, we will prove that $\triangledown_{k+1}>0$. By regarding $k$ as $k+1$, from (\ref{delta-2}) we get

\begin{equation}\label{delta-5}\begin{array}{ll}
\triangledown_{k+1}=2^{n-1}\sum\limits_{t=0}^{n}\sum\limits_{j=0}^{n}\frac{1}{t+1}\left[{n-k\choose 2j}{k\choose t-2j}-{n-k\choose 2j-1}{k\choose t-2j+1}\right].
\end{array}\end{equation}
Notice that the general term of (\ref{delta-5}) can be simplified as
\[\begin{array}{lll}
&&{n-k\choose 2j}{k\choose t-2j}-{n-k\choose 2j-1}{k\choose t-2j+1}\\
&=&\left[{n-k+1\choose 2j}-{n-k\choose 2j-1}\right]\left[{k-1\choose t-2j}+{k-1\choose t-2j-1}\right]
-\left[{n-k+1\choose 2j-1}  -  {n-k\choose 2j-2}\right]\left[{k-1\choose t-2j}  +  {k-1\choose t-2j+1}\right]\\
&  =  &  \left({n-k+1\choose 2j}{k-1\choose t-2j} -  {n-k+1\choose 2j-1}{k-1\choose t-2j+1}\right)+\left({n-k\choose 2j-2}{k-1\choose t-2j+1}-{n-k\choose2j-1}{k-1\choose t-2j}\right)\\
&&+\left({n-k+1\choose 2j}{k-1\choose t-2j-1} - {n-k\choose2j-1}{k-1\choose t-2j-1}\right)+\left({n-k\choose 2j-2}{k-1\choose t-2j}  -  {n-k+1\choose 2j-1}{k-1\choose t-2j}\right)\\
&  =  &  \left({n-k+1\choose 2j}{k-1\choose t-2j}-{n-k+1\choose 2j-1}{k-1\choose t-2j+1}\right)+\left({n-k\choose 2j-2}{k-1\choose t-2j+1}-{n-k\choose2j-1}{k-1\choose t-2j}\right)+{n-k\choose2j}{k-1\choose t-2j-1}-{n-k\choose 2j-1}{k-1\choose t-2j}\\
&  =  &  \left({n-k+1\choose 2j}{k-1\choose t-2j}-{n-k+1\choose 2j-1}{k-1\choose t-2j+1}\right)+{n-k\choose 2j-2}{k-1\choose t-2j+1}+{n-k\choose2j}{k-1\choose t-2j-1}-2{n-k\choose 2j-1}{k-1\choose t-2j}.
\end{array}\]
We can rewrite $\triangledown_{k+1}$ as follows.

\[\begin{array}{lll}
\triangledown_{k+1}&=&2^{n-1}\sum\limits_{t=0}^{n}\sum\limits_{j=0}^{n}\frac{1}{t+1}\left[{n-k\choose 2j}{k\choose t-2j}-{n-k\choose 2j-1}{k\choose t-2j+1}\right]\\
&=&2^{n-1}\sum\limits_{t=0}^{n}\sum\limits_{j=0}^{n}\frac{1}{t+1}\left[{n-k+1\choose 2j}{k-1\choose t-2j}-{n-k+1\choose 2j-1}{k-1\choose t-2j+1}\right]\\
&&+2^{n-1}\sum\limits_{t=0}^{n}\sum\limits_{j=0}^{n}\frac{1}{t+1}\left[{n-k\choose 2j}{k-1\choose t-2j-1}+{n-k\choose 2j-2}{k-1\choose t-2j+1} - 2{n-k\choose 2j-1}{k-1\choose t-2j}\right]\\
&=&\triangledown_{k}+2^{n-1}\sum\limits_{t=0}^{n}\frac{1}{t+1}\sum\limits_{j=0}^{n}\left[{n-k\choose 2j}{k-1\choose t-2j-1}+{n-k\choose 2j-2}{k-1\choose t-2j+1}-2{n-k\choose 2j-1}{k-1\choose t-2j}\right]\\
&=&\triangledown_{k}+2^{n-1}\sum\limits_{t=0}^{n}\frac{1}{t+1}\sum\limits_{j=0}^{n}\left[2{n-k\choose 2j}{k-1\choose t-2j-1}-2{n-k\choose 2j-1}{k-1\choose t-2j}\right]\\
&=&\triangledown_{k}+2^{n}\sum\limits_{t=0}^{n}\sum\limits_{j=0}^{n}\frac{1}{t+1}\left[{n-k\choose 2j}{k-1\choose t-2j-1}-{n-k\choose 2j-1}{k-1\choose t-2j}\right]\\
&=&\triangledown_{k}+2^{n}\sum\limits_{t=0}^{n}\sum\limits_{j=0}^{n}\frac{1}{t+1}\left[{n-k\choose 2j}{k-1\choose t-2j-1}-{n-k\choose 2j-1}{k-1\choose t-2j}\right].\\
\end{array}\]
By the induction hypothesis, we have $\triangledown_{k}>0$. From Lemma \ref{lem-7}, we know that the above last term are no less than zero. It implies  that $\triangledown_{k+1}=Kf(Q_{n, k+2})-Kf(Q_{n, k+1})> 0$, i.e., $Kf(Q_{n, k+1})> Kf(Q_{n, k})$ for any $1\le k< n-1$.

This completes the proof.
\end{proof}

\begin{remark}
From Theorem \ref{thm-5} we have an interesting observation that $Q_{n,1}, Q_{n,2}$,..., $Q_{n, n-1}$ have different Kirchhoff indexes. It implies that they have different spectra and so they are not isomorphic from each other.
\end{remark}

By Theorem \ref{thm-5}, we know that  the Kirchhoff index of $Q_{n, k}$ is increased as $k$ increases. It follows that $Kf(Q_{n,1})\le Kf(Q_{n,k})\le Kf(Q_{n,n-1})$ for $1\le k\le n-1$. Thus we obtain the lower  and upper bounds of $Kf(Q_{n, k})$ below.
\begin{cor}\label{cor-6}
$\sum\limits_{t=1}^{n}\frac{2^{t}-1}{t+1}\le \frac{Kf(Q_{n,k})}{2^{n-1}}\le \sum\limits_{t=1}^{n-2}\frac{2^{t}-1}{t}+3\frac{(n-2)2^{n-1}+1}{n(n-1)}$.
\end{cor}

From the above Corollary we have the following limit function for $Kf(Q_{n, k})$.

\begin{thm}\label{thm-7}
$\lim\limits_{n\to\infty}\frac{Kf(Q_{n, k})}{\frac{2^{2n}}{n+1}}=1$.
\end{thm}

\begin{proof}
By Corollary \ref{cor-6}, we have $\sum\limits_{t=1}^{n}\frac{2^{t}-1}{t+1}\le \frac{Kf(Q_{n,k})}{2^{n-1}}\le \sum\limits_{t=1}^{n-2}\frac{2^{t}-1}{t}+3\frac{(n-2)2^{n-1}+1}{n(n-1)}$.\\
Hence
\[\begin{array}{lll}
\frac{\sum\limits_{t=1}^{n}\frac{2^{t}-1}{t+1}}{\frac{2^{n+1}}{n+1}}\le \frac{Kf(Q_{n, k})}{\frac{2^{2n}}{n+1}}\le \frac{\sum\limits_{t=1}^{n-2}\frac{2^{t}-1}{t}+3\frac{(n-2)2^{n-1}+1}{n(n-1)}}{\frac{2^{n+1}}{n+1}}.\end{array}\]
Denote by $A_{n}=\frac{\sum\limits_{t=1}^{n}\frac{2^{t}-1}{t+1}}{\frac{2^{n+1}}{n+1}}$ and $B_{n}=\frac{\sum\limits_{t=1}^{n-2}\frac{2^{t}-1}{t}+3\frac{(n-2)2^{n-1}+1}{n(n-1)}}{\frac{2^{n+1}}{n+1}}=\frac{\sum\limits_{t=1}^{n-2}\frac{2^{t}-1}{t}}{\frac{2^{n+1}}{n+1}}+3[\frac{n^{2}-n-2}{4n^{2}-4n}+\frac{n+1}{n(n-1)2^{n+1}}]$. In what follows, we will show that $\lim\limits_{n\to\infty}A_{n}=\lim\limits_{n\to\infty}B_{n}=1$.
Let $x_{n}=\sum\limits_{t=1}^{n}\frac{2^{t}-1}{t+1}=\frac{1}{2}+\frac{3}{3}+\cdots+\frac{2^{n}-1}{n+1}$ and  $y_{n}=\frac{2^{n+1}}{n+1}$. We have
\begin{equation}\label{lim-eq-1}\begin{array}{lll}
 \lim\limits_{n\to\infty}\frac{x_{n}-x_{n-1}}{y_{n}-y_{n-1}}=\lim\limits_{n\to\infty}\frac{\frac{2^{n}-1}{n+1}}{\frac{(n-1)2^{n}}{n(n+1)}}=\lim\limits_{n\to\infty}\frac{n}{n-1}-\lim\limits_{n\to\infty}\frac{n}{(n-1)2^{n}}=1-0=1.\end{array}\end{equation}

By the Stolz-Ces\'{a}ro Theorem, we get\\
\[\lim\limits_{n\to\infty}A_{n}=\lim\limits_{n\to\infty}\frac{x_{n}}{y_{n}}=1.\]
Let $x_1'=x_2'=0$, and for $n\ge 3$ let $x_{n}'=\sum\limits_{t=1}^{n-2}\frac{2^{t}-1}{t}=1+\frac{3}{2}+\frac{7}{3}+\cdots+\frac{2^{n-2}-1}{n-2}$, $y_{n}'=\frac{2^{n+1}}{n+1}$. As similar as (\ref{lim-eq-1}), we have
\[\begin{array}{lll}
\lim\limits_{n\to\infty}\frac{\sum\limits_{t=1}^{n-2}\frac{2^{t}-1}{t}}{\frac{2^{n+1}}{n+1}}=\lim\limits_{n\to\infty}\frac{x_{n}'-x_{n-1}'}{y_{n}'-y_{n-1}'}=\lim\limits_{n\to\infty}\frac{n^{2}+n}{4(n^{2}-3n+2)}-\lim\limits_{n\to\infty}\frac{n(n+1)}{(n-2)(n-1)2^{n}}=\frac{1}{4}.
\end{array}\]
Thus
\[\begin{array}{ll}
\lim\limits_{n\to\infty}B_{n}&=\lim\limits_{n\to\infty}\left(\frac{\sum\limits_{t=1}^{n-2}\frac{2^{t}-1}{t}}{\frac{2^{n+1}}{n+1}}+3[\frac{n^{2}-n-2}{4n^{2}-4n}+\frac{n+1}{n(n-1)2^{n+1}}]\right)\\
&=\lim\limits_{n\to\infty}\frac{\sum\limits_{t=1}^{n-2}\frac{2^{t}-1}{t}}{\frac{2^{n+1}}{n+1}}+3\lim\limits_{n\to\infty}\frac{n^{2}-n-2}{4n^{2}-4n}+3\lim\limits_{n\to\infty}\frac{n+1}{n(n-1)2^{n+1}}\\
&=\frac{1}{4}+\frac{3}{4}-0=1
\end{array}\]
By the Squeeze Theorem, we get\\
\[\begin{array}{lll}
\lim\limits_{n\to\infty}\frac{Kf(Q_{n, k})}{\frac{2^{2n}}{n+1}}=\lim\limits_{n\to\infty}A_{n}=\lim\limits_{n\to\infty}B_{n}=1.\end{array}\]
This completes the proof.
\end{proof}

\begin{remark}
Although we get explicit formula of $Kf(Q_{n, k})$ when $k=1$ or $n-1$ in Corollary \ref{cor-4} and Corollary \ref{cor-5}, respectively. The calculation of the general representation of $Kf(Q_{n, k})$ given in Theorem \ref{thm-4} is more complex. Fortunately, Theorem \ref{thm-7} provides a simple uniform approximation function for $Kf(Q_{n, k})$ which is also independent of $k$, it means that $Kf(Q_{n, k})$ can be replaced with $\frac{2^{2n}}{n+1}$ if $n$ is large enough.
\end{remark}

\end{document}